\documentclass[reqno]{amsart}

\usepackage{amsmath} 
\usepackage{amsfonts}
\usepackage{amssymb}
\usepackage{amstext}
\usepackage{amsbsy}
\usepackage{amsopn}
\usepackage{amsthm}
\usepackage{amsxtra}
\usepackage{graphicx}
\usepackage{caption}
\usepackage{subcaption}
\usepackage{color}
\usepackage{hyperref}
\usepackage{enumerate}
\usepackage{amscd}
\usepackage{tikz}

\newtheorem{theorem}{Theorem}[section]
\newtheorem{lemma}[theorem]{Lemma}
\newtheorem{corollary}[theorem]{Corollary}

\newcommand{\CA}{\mathcal{A}}
\newcommand{\CL}{\mathcal{L}}

\newcommand{\R}{\mathbb{R}}
\newcommand{\Z}{\mathbb{Z}}
\newcommand{\Q}{\mathbb{Q}}
\newcommand{\N}{\mathbb{N}}
\newcommand{\T}{\mathbb{T}}
 
\renewcommand{\P}{\mathcal{P}}

\newcounter{note}

\begin{document}

\title[Boshernitzan's condition]{Boshernitzan's condition, factor complexity, and an application}
\author{Van Cyr}
\address{Bucknell University, Lewisburg, PA 17837 USA}
\email{van.cyr@bucknell.edu}
\author{Bryna Kra}
\address{Northwestern University, Evanston, IL 60208 USA}
\email{kra@math.northwestern.edu}

\subjclass[2010]{37B10 (primary), 37B40, 35J10, 37A35}
\keywords{subshift, complexity, Schr\"odinger operator, uniquely ergodic}


\thanks{The first author thanks Northwestern for its hospitality while this work was completed and the second author was partially supported by NSF grant DMS-1800544.}

\begin{abstract}
Boshernitzan found a decay condition on the measure of cylinder sets that implies unique ergodicity for minimal subshifts.  Interest in the properties of subshifts satisfying this condition has grown recently, due to a connection with the study of discrete Schr\"odinger operators.  Of particular interest is the question of how restrictive Boshernitzan's condition is.  While it implies zero topological entropy, our main theorem shows how to construct minimal subshifts satisfying the condition  whose factor complexity grows faster than any pre-assigned subexponential rate.  
As an application, via a theorem of Damanik and Lenz, we show that there is no subexponentially growing sequence 
for which the spectra of all discrete Schr\"odinger operators associated with subshifts whose complexity grows faster than
the given sequence, have only finitely many gaps.
\end{abstract}

\maketitle

\section{Boshernitzan's complexity conditions}

For a symbolic dynamical system $(X,\sigma)$, many of the isomorphism invariants we have are statements about the growth rate of the {\em word complexity function $P_X(n)$}, 
which counts the number of distinct cylinder sets determined by words of length $n$ 
having nonempty intersection with $X$.  For example, the exponential growth of $P_X(n)$ is the topological entropy of $(X,\sigma)$, while the linear growth rate of $P_X(n)$ gives an invariant to begin distinguishing between zero entropy systems.  Of course there are different senses in which the growth of $P_X(n)$ could be said to be linear and different invariants arise from them.  For example, one can consider systems with {\em linear limit inferior growth}, meaning
 $\liminf_{n\to\infty}P_X(n)/n < \infty$,  or  the stronger condition of {\em linear limit superior growth}, meaning $\limsup_{n\to\infty}P_X(n)/n < \infty$. (There exist systems satisfying the first condition but not satisfying the second.) 
 
Under the assumption of linear limit inferior growth, 
 and with a further hypothesis that the system $(X,\sigma)$ is minimal, 
 Boshernitzan~\cite{boshernitzan} showed that the system only supports finitely many $\sigma$-invariant ergodic probability measures. 
 Boshernitzan also considered another version of linear complexity on a minimal shift, studying 
 {\em linear measure growth}, also referred to in the literature (see for example~\cite{DL}) as {\em condition (B)}: if $\mu$ is a $\sigma$-invariant Borel probability measure on $X$, assume that there exists a sequence 
 of integers $n_k \to \infty$ such that $$\inf_k \min_{|w|=n_k} n_k\mu([w]) > 0,$$  
 where $\mu([w])$ denotes the measure of the cylinder set determined by the word $w$ and $|w|$ denotes the length of the word $w$.  Boshernitzan showed that linear measure growth implies that the minimal subshift $(X,\sigma)$ is uniquely ergodic.  Another consequence of linear measure growth, for word complexity, is that  
 $\liminf_{n\to\infty}P_X(n)/n$ is finite.  
 
 Each of these three linear complexity assumptions, linear limit inferior growth, linear limit superior growth, and linear measure growth, immediately implies that the associated system has zero topological entropy.  It is natural to ask which of these conditions imply any of the others.  One of our main results is that while linear measure growth implies linear limit inferior growth, it does not imply linear limit superior growth.  In fact, we show linear measure growth is flexible enough that examples satisfying it can be constructed with limit superior growth faster than any pre-assigned subexponential growth rate. 

A second motivation for the construction we give comes from a question on the spectra of discrete Schr\"odinger operators that arise from a subshift.  If $(X,\sigma)$ is a shift, then each $x = (x_n)_{n\in\Z}\in X$ defines a discrete Schr\"odinger operator $H_x\colon\ell^2(\Z)\to\ell^2(\Z)$ by 
$$ 
(H_xu)(n):=u(n+1)+u(n-1)+x_nu(n) 
$$ 
(and $x$ is called the {\em potential function} for this operator).  Characterizing the spectra of discrete Schr\"odinger operators is an active field of study (e.g.,~\cite{DL,Avila-Krikorian,Bourgain-Jitomirskaya}) and we refer the reader to~\cite{Damanik2, Damanik} for excellent surveys on the theory of discrete Schr\"odinger operators associated 
with symbolic systems.  For operators built in this way, the dynamical properties of $(X,\sigma)$ can influence the spectral properties of $H_x$ for any $x\in X$.  When $(X,\sigma)$ is minimal, Damanik (personal communication) asked whether the condition that $h_{top}(X)>0$ implies that the spectrum of $H_x$ can have only finitely many gaps.  Our example shows that the assumption of positive 
entropy in this question cannot be relaxed to just ask that $P_X(n)$ grow ``nearly exponentially'' infinitely often: for any subexponential rate $\{a_n\}_{n=1}^{\infty}$ our example, via a theorem of Damanik and Lenz~\cite{DL}, gives a Schr\"odinger operator whose spectrum has infinitely many gaps and whose complexity is larger than $\{a_n\}_{n=1}^{\infty}$ infinitely often.

We turn to stating our main theorem.  For a word $w$ in the language of a subshift $(X, \sigma)$, we denote the cylinder set starting at zero it determines by $[w]_0^+$ and we denote the words of length $n$ in the language of the subshift by $\mathcal L_n(X)$ (for further discussion of the definitions, see Section~\ref{sec:background-symbolic}): 
\begin{theorem}
\label{th:main}
Let $\{a_n\}_{n=1}^{\infty}$ be a sequence of positive integers 
satisfying 
\begin{equation}\label{eq:growth} 
\limsup_{n\to\infty}\frac{1}{n}\log a_n=0. 
\end{equation} 
There exists a minimal and uniquely ergodic subshift $(Y, \sigma)$ such that 
$$
\limsup_{n\to\infty}\frac{P_Y(n)}{a_n}=\infty 
$$
and such that the unique invariant measure $\mu$ has the property that there is 
a sequence $\{n_k\}_{k=1}^{\infty}$ satisfying 
\begin{equation}
\label{eq:prop}
\liminf_{k\to\infty}\min\left\{\mu([w]_0^+)\cdot n_k\colon w\in\mathcal{L}_{n_k}(Y)\right\}>0. 
\end{equation}
\end{theorem}
The hypothesis in this theorem is a type of subexponential growth on the sequence $\{a_n\}_{n=1}^{\infty}$ and the constructed system is a zero entropy system satisfying the Boshernitzan condition while the factor complexity grows faster than the given sequence.  
To prove the theorem, it suffices to show that the system $(Y, \sigma)$ supports a measure $\mu$ satisfying the property~\eqref{eq:prop}, as it then follows from Boshernitzan~\cite[Theorem 1.2]{boshernitzan2} that the system is uniquely ergodic.  

An immediate corollary of Theorem~\ref{th:main}, combined with a theorem of Damanik and Lenz~\cite[Theorem 2]{DL}, is the following: 
\begin{corollary}
Let $\{a_n\}_{n=1}^{\infty}$ be a sequence of positive integers which grows subexponentially in the sense of~\eqref{eq:growth}.  There exists a Cantor set $\Sigma\subset\mathbb{R}$, of Lebesgue measure zero, and a minimal subshift $(Y,\sigma)$ such that 
$$
\limsup_{n\to\infty}\frac{P_Y(n)}{a_n}=\infty 
$$
and for every $y\in Y$ the discrete Schr\"odinger operator $H_y\colon\ell^2(\Z)\to\ell^2(\Z)$ given by 
$$ 
(H_yu)(n):=u(n-1)+u(n+1)+y_nu(n) 
$$ 
and the spectrum of $H_y$ is exactly $\Sigma$. 
\end{corollary} 

\subsection*{Acknowledgment}
We thank David Damanik for bringing this question to our attention and for his helpful remarks during the preparation of this paper.

\section{Background}

\subsection{Symbolic systems} 
\label{sec:background-symbolic}
We work over the alphabet $\CA = \{0,1\}\subset\R$ and consider $\CA^\Z$.  We denote $x\in\CA^\Z$ as $x = (x_n)_{n\in\Z}$ and 
we endow  $\CA^\Z$ with the topology induced by the metric $d(x, y) = 2^{-\inf\{|i|\colon x_i \neq y_i\}}$.  
The {\em left shift} $\sigma \colon \CA^\Z\to\CA^\Z$ is defined by $(\sigma x)_n = x_{n+1}$ for all $n\in\Z$.
If $X\subset\CA^\Z$ is closed and $\sigma$-invariant, then $(X, \sigma)$ is a {\em subshift}. 

If $w = (a_{-n}, \ldots, a_0, \ldots, a_n)\in\CA^{2n+1}$, then the {\em central cylinder set 
$[w]_0$ determined by $w$ }is defined to be 
$$\{x\in X\colon x_j = a_j \text{ for } j=-n, \ldots, n\}
$$
and the {\em one-sided cylinder set $[w]_0^+$ determined by $w$} is defined to be 
$$\{x\in X\colon x_j = a_j \text{ for } j=0, \ldots, n\}.  
$$
If $(X, \sigma)$ is a subshift and $n\in\N$, the {\em words $\CL_n(X)$  of length $n$} 
are defined to be the collection of all $w\in\CA^n$ such that $[w]_0^+\neq\emptyset$, 
and the {\em language $\CL(X)$} of the subshift is the union of all the words: 
$$\CL(X) = \bigcup_{n=1}^\infty \CL_n(X). $$
If $w\in\CL(X)$ is a word, we say that {\em  $u\in\CL(X)$ is a subword of $w$} if $w = w_1uw_2$ for some (possibly empty) words $w_1, w_2\in\CL(X)$.  

For a subshift $(X, \sigma)$, the {\em word complexity $P_X(n)\colon \N\to\N$} is defined to be the number of words of length $n$ in the language: 
$$
P_X(n) = |\CL_n(X)|.
$$

\subsection{Well approximable irrationals}
\label{sec:approximations}
A key ingredient in our construction is the following theorem of V.~S\'os~\cite{sos} (formerly known as the Steinhaus Conjecture). 
\begin{theorem}[The Three Gap Theorem]
 Assume $\alpha\in\R\setminus\Q$ and $n\in\N$, the partition of  the unit circle $\T = \R/\Z$
 determined by the points $\{0, \alpha, 2\alpha, \ldots, (n-1)\alpha\}$, with all points taken $\pmod 1$.  
 Then the subintervals determined by this partition have at most three distinct lengths, and when there are three distinct length, the largest length is the sum of the other two. 
\end{theorem} 
Given an integer $n\geq 1$ and irrational $\alpha$, we refer to the partition determined by the points  $\{0, \alpha, 2\alpha, \ldots, (n-1)\alpha\}$ of the unit circle as the {\em $n$-step partition}, and make use of it for well chosen $\alpha$.  
An irrational real number $\alpha$ is {\em well approximable} if there exists a sequence $\{n_k\}_{k=1}^{\infty}$ of integers 
such that for each $n_k$, the associated $n_k$-step partition 
in the Three Gap Theorem has three distinct lengths and the ratio of the smallest to the largest length in such a partition tends to zero as $k\to\infty$. (This  sequence is obtained as the denominators in the regular continued fraction expansion of $\alpha$, and this can be rephrased as unbounded partial quotients.) Furthermore, 
we can choose the sequence $\{n_k\}_{k=1}^{\infty}$ such that 
the smallest length present in the $n_k$-step partition is not present for in the $(n_k-1)$-step partition. 
 
An irrational that is not well approximable is said to be {\em badly approximable}, 
and the set of badly approximable reals has Lebesgue measure zero. 
Notice that if $\alpha$ is well approximable and $\{n_k\}_{k=1}^{\infty}$ is the associated sequence, then the $(n_k-1)$-step  partition in the Three Gap Theorem has only two distinct lengths and the ratio of their lengths tends to $1$ as $k\to\infty$.

\subsection{Sturmian systems}
\label{sec:sturmian}
To make use of the approximations determined by the Three Gap Theorem, we use Sturmian sequences.  
To define this notion, let $\alpha$ be an irrational real number and consider the partition $\mathcal{P}=\{[0,\alpha),[\alpha,1)\}$ of $[0,1)$ and let $T_\alpha$ denote the rotation $T(x) = x+\alpha\pmod 1$.  For any $x\in[0,1)$ and each $n\in\Z$, define 
$$ 
c_n(x)=\left\{\begin{tabular}{ll} 
$0$ & if $x+n\alpha\pmod1 \in [0, \alpha)$; \\ 
$1$ & otherwise. 
\end{tabular}\right. 
$$ 
Let $X_{\alpha}\subseteq\{0,1\}^{\Z}$ be closure of the set of all sequences of the form 
$$ 
(\dots,c_{-2}(x),c_{-1}(x),c_0(x),c_1(x),c_2(x),\dots). 
$$ 
Then $X_{\alpha}$ is called the {\em Sturmian shift} with {\em rotation angle} $\alpha$.  A classical fact is that 
the system $(X_{\alpha}, \sigma)$ is minimal, uniquely ergodic, and $P_{X_{\alpha}}(n)=n+1$ for all $n\in\N$. 

Moreover, words $w\in\mathcal{L}_{X_{\alpha}}(X)$ correspond to the cells of 
$\bigvee_{i=0}^{|w|-1}T_\alpha^{-i}\mathcal{P}$, and with respect to the unique invariant measure 
$\nu_{\alpha}$, the measure of the cylinder set $[w]_0^+$ is the Lebesgue measure $\lambda$ of the cell of 
$\bigvee_{i=0}^{|w|-1}T_\alpha^{-i}\mathcal{P}$ corresponding to $w$. 
In other words, there is a bijection 
\begin{equation}
\label{eq:biject}
\varphi_{n}\colon\mathcal{L}_n(X_{\alpha})\to\P_{n}
\end{equation}
such that 
for any word $w\in\mathcal{L}_n(X)$, we have 
$$\nu_i([w]_0^+)=\lambda(\varphi_{n}(w)).$$

 In view of the discussion in Section~\ref{sec:approximations}, if $\alpha$ is well approximable, there exists a sequence $\{n_k\}_{k=1}^{\infty}$ such that 
\begin{equation}
\label{eq:well-approx-use} 
\lim_{k\to\infty}\frac{\min\{\nu_{\alpha}([w]_0^+)\colon w\in\mathcal{L}_{n_k-1}(X_{\alpha})\}}{\max\{\nu_{\alpha}([w]_0^+)\colon w\in\mathcal{L}_{n_k-1}(X_{\alpha})\}}=1. 
\end{equation}

Recall that $(X, \sigma)$ is {\em uniquely ergodic} if there exists a unique Borel probability $\sigma$-invariant measure on $X$.  Recasting this definition in terms of the language, the subshift $(X, \sigma)$ is uniquely ergodic if and only if for any $w\in\mathcal{L}(X)$, there exists $\delta\geq0$ such that for any $\varepsilon>0$ there 
is an integer $N\geq 1$ with the property that for all $u\in\mathcal{L}(X)$ with $|u|\geq N$,  we have 
$$ 
\left|\frac{\#\text{ of occurrences of $w$ as a subword of $u$}}{|u|}-\delta\right|<\varepsilon. 
$$ 
In this case, $\delta$ is the measure of the cylinder set $[w]_0^+$ with respect to the unique invariant measure on $X$.

\section{The construction}

We construct a minimal subshift $X\subseteq\{0,1\}^{\Z}$ such that 
\begin{equation}\label{eq1} 
\limsup_{n\to\infty}\frac{P_X(n)}{n}=\infty 
\end{equation} 
and for which there exists an invariant measure $\mu$ supported on $X$ and a sequence $\{n_k\}_{k=1}^{\infty}$ satisfying 
\begin{equation}\label{eq2}
\liminf_{k\to\infty}\min\left\{\mu([w]_0^+)\cdot n_k\colon w\in\mathcal{L}_{n_k}(X)\right\}>0. 
\end{equation} 

\subsection{Setup}
\label{previous}
We fix $\varepsilon = 1/8$ (any value $\leq 1/8$ suffices) and choose a well approximable real numbers $\alpha$ satisfying
\begin{equation}\label{eq:alpha}
\left|\frac{1}{2}-\alpha\right|<\varepsilon. 
\end{equation} 

Let $X_{\alpha}$ denote the Sturmian shift with rotation angle $\alpha$ and let $\nu$ denote the (unique) invariant measure supported on $X_{\alpha}$ (see Section~\ref{sec:sturmian}).  For each $n\in\N$, 
let $\P_{n}$ denote the partition of $[0,1)$ into subintervals whose endpoints are given by the set 
$$ 
\{0,\alpha,2\alpha, \dots,(n-1)\alpha\}, 
$$ 
where, as usual, all points are taken in $[0,1)$, meaning modulo $1$.  

Using~\eqref{eq:well-approx-use}  derived from the well approximability of $\alpha$, 
there exists $n\in\N$ satisfying
\begin{equation}
\label{eq:exist-n}
\frac{\lambda(\text{shortest subinterval in }\P_{n})}{\lambda(\text{longest subinterval in }\P_{n})}>1-\varepsilon
\end{equation} 
(in fact there exist infinitely many such $n$).  
The partition $\P_{n}$ is obtained from the partition 
$\P_{n-1}$ by subdividing one of the subintervals in $\P_{n-1}$ into two pieces. 
 Thus the length of the longest subinterval in $\P_{n-1}$ is at most twice the length of the longest subinterval in $\P_{n}$.  Similarly the length of the shortest subinterval in $\P_{n-1}$ is at least as long as the length of the shortest subinterval in $\P_{n}$.  Therefore we also have 
$$ 
\frac{\lambda(\text{shortest subinterval in }\P_{n-1})}{\lambda(\text{longest subinterval in }\P_{n-1})}\geq\frac{\lambda(\text{shortest subinterval in }\P_{n})}{2\cdot\lambda(\text{longest subinterval in }\P_{n})}>\frac{1}{2}-\frac{\varepsilon}{2}.  
$$ 

We are now ready to begin our construction. 
\begin{equation}\label{eq:fixed}
\text{Fix some $n$ satisfying~\eqref{eq:exist-n}}
\end{equation}
and let  
$\varphi_{n-1}\colon\mathcal{L}_{n-1}(X_{\alpha})\to\P_{n-1}$ and $\varphi_{n}\colon\mathcal{L}_{n}(X_{\alpha})\to\P_{n}$  denote the bijections defined in~\eqref{eq:biject}.  Then
\begin{equation}\label{eq:longer} 
\frac{\min\{\nu([w]_0^+)\colon w\in\mathcal{L}_{n}(X_{\alpha})\}}{\max\{\nu([w]_0^+)\colon w\in\mathcal{L}_{n}(X_{\alpha})\}}>1-\varepsilon 
\end{equation} 
and 
\begin{equation}\label{eq:shorter} 
\frac{\min\{\nu([w]_0^+)\colon w\in\mathcal{L}_{n-1}(X_{\alpha})\}}{\max\{\nu([w]_0^+)\colon w\in\mathcal{L}_{n-1}(X_{\alpha})\}}>\frac{1}{2}-\frac{\varepsilon}{2}. 
\end{equation} 

Since $X_{\alpha}$ is uniquely ergodic, we can choose 
$N\in\N$ such that for any $m\geq N$ and any word $w\in\mathcal{L}_{n-1}(X_{\alpha})\cup\mathcal{L}_{n}(X_{\alpha})$ and any word $u\in\mathcal{L}_m(X_{\alpha})$,  we have 
\begin{equation}\label{eq:frequency} 
\left|\frac{\#\text{ of occurrences of $w$ as a subword of $u$}}{|u|}-\nu([w]_0^+)\right|<\varepsilon. 
\end{equation} 

Since $X_{\alpha}$ is Sturmian, we have $P_{X_{\alpha}}(m)=m+1$ for all $m\in\N$.  Equivalently, this mean that  $P_{X_{\alpha}}(m+1)=P_{X_{\alpha}}(m)+1$ for all $n\in\N$.  In particular, for all $n$ there is a unique word $w\in\mathcal{L}_m(X_{\alpha})$ for which both $w0$ and $w1$ are elements of $\mathcal{L}_{m+1}(X_{\alpha})$.  
Let $w\in\mathcal{L}_{N}(X_{\alpha})$ be the unique word with this property.  Note that for any $m\geq N$, the unique word in $\mathcal{L}_m(X_{\alpha})$ with this property has $w$ as its rightmost subword (of length $|w|$).

Since $X_{\alpha}$ is minimal, 
all sufficiently long words in $\mathcal{L}(X_{\alpha})$ contain every word of length $|w|+1$ 
as a subword and there is a uniform gap $g$ (which depends only on $|w|$) 
between consecutive occurrences of any word in $\mathcal{L}_{|w|+1}(X_{\alpha})$. 
 Let $m\geq N+3g+3|w|$ be sufficiently large that the unique word $u\in\mathcal{L}_m(X_{\alpha})$ for which both $u0$ and $u1$ are in $\mathcal{L}_{m+1}(X_{\alpha})$, has this property.  
 Then the rightmost subword of $u$ of length $|w|$ is $w$ and there is an occurrence of $w0$ within distance $g$ of the left edge of $u$.  
 Define $a$ to be the subword of $u$ that begins with the leftmost occurrence of $w0$ and ends just before the rightmost occurrence of $w$ 
 (meaning we remove the rightmost $|w|$ letters of $u$ to obtain the end of the word $a$).  
 Note that $|a|\geq\max\{2g+2|w|,N\}$ and so~\eqref{eq:frequency} holds for all words in $\mathcal{L}_{|w|}(X_{\alpha})$ and $u=a$ (because its length is at least $N$).  Since $|a|\geq2g+2|w|$, every word in $\mathcal{L}_{|w|}(X_{\alpha})$ occurs as a subword of $a$.  Moreover every subword of $aa$ of length $|w|$ is an element of $\mathcal{L}(X_{\alpha})$, 
since $aw0\in\mathcal{L}_{|aw0|}(X_{\alpha})$ and the leftmost subword of length $|w0|$ in $a$ is $w0$.  Since $X_{\alpha}$ is aperiodic, there exists an integer $e$ such that 
$$ 
A:=\underbrace{aa\cdots a}_{e\text{ times}}\notin\mathcal{L}(X_{\alpha}). 
$$ 
Let $n\geq3|A|+3g+3|w|$ and let $v\in\mathcal{L}_n(X_{\alpha})$ be the unique word for which $v0$ and $v1$ are both elements of $\mathcal{L}_{n+1}(X_{\alpha})$.  
Let $b$ be the subword of $v$ that begins at the leftmost occurrence of $w1$ and ends just before the rightmost occurrence of $w$.  
Then $|b|\geq\max\{2g+2|w|,3|A|\}\geq\max\{2g+2|w|,N\}$ and so~\eqref{eq:frequency} holds for the words $w$ and $u=b$ and every word in $\mathcal{L}_{|w|}(X_{\alpha})$ occurs as a subword of $b$.  
Moreover every subword of length $|w|$ that occurs in $ab$, $ba$, and $bb$ is in $\mathcal{L}_{|w|}(X_{\alpha})$, since $bw0, bw1\in\mathcal{L}(X_{\alpha})$, the leftmost subword of $a$ is $w0$ and the leftmost subword of $b$ is $w1$.  Finally we define two words: 
	\begin{eqnarray*} 
	x&:=&\underbrace{AA\cdots A}_{|b|\text{ times}} \\ 
	y&:=&\underbrace{bb\cdots b}_{|A|\text{ times}}. 
	\end{eqnarray*} 
By construction both of these words are periodic and we let $p$ denote the minimal period of $x$ and let $q$ denote the minimal period of $y$.  These words have the following properties: 
	\begin{enumerate} 
	\item $|x|=|y|$; 
	\item $A$ does not occur as a subword of $y$ (since $y\in\mathcal{L}(X_{\alpha})$ 
	but $A\notin\mathcal{L}(X_{\alpha})$); 
	\item $x$ does not occur as a subword of $yy$ (because $|y|\geq3|A|$ and 
	if $x$ occurred in $yy$ it would force an occurrence of $A$ in $y$); 
	\item $y$ does not occur as a subword of $xx$ (again because such an occurrence would force 
	and occurrence of $A$ in $y$); 
	\item $x$ occurs exactly once as a subword of $xy$.  Namely, $x$ cannot overlap $y$ by at least $|A|$ symbols without forcing an 
	occurrence of $A$ in $y$, and 
	so $x$ must overlap $x$ on more than $|A|\geq2p$ 
	symbols.  
	This means that the occurrence of $x$ (the subword) has to be offset from the beginning of 
	$xy$ by a multiple of $p$, 
	and so if this occurrence of $x$ overlaps $y$ by at least $|w0|$ many 
	symbols, then since $|x|$ is a multiple of $p$, this 
	implies that $y$ has to begin with the word 
	$w0$, but it begins with $w1$, a contradiction.  This means that $x$ (the subword) overlaps $y$ on at 
	most $|w|$ symbols, and so a portion of the start of $w$ (the leftmost subword of $y$) whose length is a multiple of the minimal period of $x$ matches the subword of the same length at the end of $x$.
	\end{enumerate} 
	
Next we define two more words, $s$ and $t$, as follows: 
	\begin{eqnarray*} 
	s&:=&xxyxx; \\ 
	t&:=&xxyyx. 
	\end{eqnarray*} 
Note that all of the words $ss$, $st$, $ts$, and $tt$ contain $xxxy$ at least once as a subword.  Consider where such a subword could occur: 
	\begin{eqnarray*} 
	ss&:=&xxyxxxxyxx; \\ 
	st&:=&xxyxxxxyyx; \\ 
	ts&:=&xxyyxxxyxx; \\ 
	tt&:=&xxyyxxxyyx. 
	\end{eqnarray*} 
We analyze where it can occur in $ss$, 
and the analysis for the other three cases is similar.  
Since $y$ does not occur as a subword of $xx$, the prefix $xxx$ (in $xxxy$) cannot completely overlap the leftmost $y$ in $ss$. 
This means that the farthest to the left that this prefix can occur is if it begins one letter after the beginning of the leftmost $y$ in $ss$.  But since the word $y$ in $xxxy$ cannot be completely contained in the central $xxxx$ of $ss$, the farthest to the left $xxxy$ could occur in $ss$ is to have the $y$ at least partially overlap the rightmost $y$ in $ss$.  Also, since the only place in $xy$ that $x$ can occur is at the leftmost edge, the $y$ in $xxxy$ cannot occur anywhere farther to the right in $ss$ than the rightmost $y$ (otherwise it would force an occurrence of $x$ in $xy$ which would guarantee that one of the subwords $xx$ in $xxxy$ exactly overlaps the rightmost $xy$ in $ss$, which is impossible since $x\neq y$).  Therefore any occurrence of $xxxy$ in $ss$ must have the $y$ in $xxxy$ partially overlap the rightmost $y$ in $ss$, but not extend any farther to the right than this occurrence of $y$.  
 This means the leftmost $x$ in $xxxy$ occurs as a subword of the central $xxxx$ in $ss$, and so it is a multiple of $p$ (the minimal period of the bi-infinite word $\cdots xxxx\cdots$) from the right edge of the central $xxxx$ in $ss$.  
 This means that the $y$ in $xxxy$ overlaps the rightmost $y$ in $ss$ and is offset from the right edge of $y$ by a multiple of $p$.  If this multiple is zero, then we are done.  Otherwise, the $y$ in $xxxy$ partially overlaps the right edge of $xxxx$ by a multiple of $p$ and overlaps the rightmost $y$ in $ss$ by this same multiple of $p$.  Therefore this multiple of $p$ is a period of $y$ and since the leftmost edge of $y$ (of this length) agrees with the $p$-periodic word $xxxx$, then the entire word $y$ is also $p$-periodic and $y=x$, a contradiction.

\begin{lemma} 
Any element $z\in\{0,1\}^{\Z}$ that can be written as a bi-infinite concatenation of the words $s$ and $t$ can be written in a unique way as such a concatenation. 
\end{lemma} 
A shift with this property is sometimes known as a {\em uniquely decipherable} coded shift. 

\begin{proof} 
Note that $s$ and $t$ have the same lengths and are not the same word.  We have already noted that the word $xxxy$ occurs in each of $ss$, $st$, $ts$, and $tt$ and, moreover, it occurs exactly once in each such word.  If $z$ can be written as a bi-infinite concatenation of the words $s$ and $t$, then there must be an occurrence of $xxxy$ within distance $|ss|$ of the origin.  Choose a way to write $z$ as a concatenation of $s$ and $t$ and mark the locations in $\Z$ where this choice places the beginnings of these words.  Let $0\leq d<|s|$ be the smallest non-negative integer that lies in this set.  Find an occurrence of $xxxy$ within distance $|ss|$ of the origin.  Since $|xxxy|<|s|=|t|$, this occurrence must be contained in one of the words $ss$, $st$, $ts$, or $tt$ that begins from our marked set of integers.  But $xxxy$ occurs exactly once in any such word and its location always places $y$ exactly $3|x|$ symbols from the right of whichever of $ss$, $st$, $ts$, or $tt$ it occurs in.  This allows us to determine where the marked integers in this occurrence of $ss$, $st$, $ts$, or $tt$ are located.  This allows us to read off the sequence of words $s$ and $t$ that were concatenated to produce $z$ by starting from one of the marked integers and looking at blocks of size $|s|$ moving to the right and left. 

Thus, once we find an occurrence of $xxxy$ within distance $|ss|$ of the origin in $z$, the locations (in $\Z$) where the words $s$ and $t$ begin is determined, and once these locations are determined, the bi-infinite sequence of $s$ and $t$ is also determined.  In other words, there is a unique way to write $z$ as such a bi-infinite concatenation. 
\end{proof} 

\begin{lemma}\label{lem:meas-concat} 
Let $Y\subseteq\{0,1\}^{\Z}$ be the subshift consisting of all elements of $\{0,1\}^{\Z}$ that can be written as bi-infinite concatenations of the words $s$ and $t$.  Let $Z\subseteq Y$ be any subshift of $Y$ and let $\mu$ be any $\sigma$-invariant probability measure on $Z$.  
Recall that $n\in\N$ is defined in~\eqref{eq:fixed}.  
Then $\mathcal{L}_{n}(Z)\cup\mathcal{L}_{n-1}(Z)=\mathcal{L}_{n}(X_{\alpha})\cup\mathcal{L}_{n-1}(X_{\alpha})$ and we have 
$$ 
\frac{\min\{\mu([w]_0^+)\colon w\in\mathcal{L}_{n}(Z)\}}{\max\{\mu([w]_0^+)\colon w\in\mathcal{L}_{n}(Z)\}}>1-2\varepsilon 
$$ 
and 
$$ 
\frac{\min\{\mu([w]_0^+)\colon w\in\mathcal{L}_{n-1}(Z)\}}{\max\{\mu([w]_0^+)\colon w\in\mathcal{L}_{n-1}(Z)\}}>\frac{1}{2}-\varepsilon. 
$$ 
\end{lemma} 
\begin{proof} 
By construction of the words $a$ and $b$ (from which $s$ and $t$ are built) our claim about equality of the languages holds.  The claim about measure follows from~\eqref{eq:longer},~\eqref{eq:shorter}, and~\eqref{eq:frequency} combined with the fact that $a,b\in\mathcal{L}(X_{\alpha})$ and $|a|,|b|\geq N$. 
\end{proof} 

\begin{lemma} 
\label{lemma:top}
$h_{top}(Y)=\log(2)/|s|>0$. 
\end{lemma} 
\begin{proof} 
The number of distinct words whose length $n$ is any particular multiple of $2|s|$ is at least $|s|\cdot 2^{n/|s|}$ and at most $4|s|\cdot 2^{n/|s|}$ since a word of this length must contain an occurrence of $xxxy$, which tells us how to parse it as a subword of a concatenation of the words $s$ and $t$ (other than perhaps the leftmost and rightmost words in the concatenation). 
\end{proof} 

We fix a subexponentially growing sequence $\{a_n\}_{n=1}^{\infty}$.  It follows from Lemma~\ref{lemma:top} that 
\begin{equation}
\label{eq:grow}
P_{Y}(m)>a_m
\end{equation} 
for all but finitely many $m\in\N$.  Thus, 
\begin{equation}\label{eq:bound}
\text{given the subexponentially growing sequence $\{a_m\}_{m=1}^{\infty}$ and a bound $B\geq|s|$,}
\end{equation}
we can choose  some $m>B$ such that $P_Y(m)>a_m$, and then fix two (distinct) words $u,v\in\mathcal{L}(Y)$, of equal length, that each contain every element of $\mathcal{L}_{m}(Y)$ as a subword.  Finally define 
	\begin{eqnarray} 
	\label{eq:0*}
	0_*&:=&uuvuuvuuvvvvuuvuuvuuv \ ; \\ 
	\label{eq:1*}
		1_*&:=&uuvuuvuuvvvvvvvuuvuuv.  
	\end{eqnarray} 
Arguing as with words that can be written as bi-infinite concatenations of $s$ and $t$, observe that any element of $\{0,1\}^{\Z}$ that can be written as a bi-infinite concatenation of $0_*$ and $1_*$ can be written in a unique way as such a concatenation.  Let $Z$ be the subshift consisting of all elements of of $\{0,1\}^{\Z}$ that can be written as a bi-infinite concatenation of the word $0_*$ and $1_*$.  Then $Z\subseteq Y$, meaning Lemma~\ref{lem:meas-concat} applies to any $\sigma$-invariant probability measure on $Z$.  Furthermore, by~\eqref{eq:grow}, we have that 
\begin{equation}
\label{eq:blank}
P_{Z}(|0_*|)>b_{|0_*|}.
\end{equation}

\subsection{Inflated subshifts} 

\begin{lemma}\label{lem:induction} 
Let $Z\subseteq\{0,1\}^{\Z}$ be a subshift.  Assume there exists an integer $m>1$ and positive constants $C_1 < C_2$ 
such that for any ergodic measure $\mu$ supported on $Z$, we have
$$ 
C_1<\frac{\min\{\mu([w]_0^+)\colon w\in\mathcal{L}_m(Z)\}}{\max\{\mu([w]_0^+)\colon w\in\mathcal{L}_m(Z)\}}<C_2. 
$$ 
Assume that $\beta_0,\beta_1\in\{0,1\}^*$ are two words of equal length such that for any $x\in\{0,1\}^{\Z}$ that can be written as a bi-infinite concatenation of the words $\beta_0$ and $\beta_1$, there is a unique way to write it as such a concatenation.  
Further assume that there is a word $v$ that appears exactly once in each of the concatenations  
$\beta_0\beta_0$, $\beta_0\beta_1$, $\beta_1\beta_0$, and $\beta_1\beta_1$.
Let $X\subseteq\{0,1\}^{\Z}$ be the subshift consisting of all elements of $\{0,1\}^{\Z}$ that can be written as a (unique) bi-infinite concatenation 
$$ 
\cdots\beta_{i_{-2}}\beta_{i_{-1}}\beta_{i_0}\beta_{i_1}\beta_{i_2}\cdots 
$$ 
where $(\dots,i_{-2},i_{-1},i_0,i_1,i_2,\dots)\in Z$.  
If $\mu$ is any ergodic measure supported on $X$, then 
$$ 
\frac{C_1}{4}<\frac{\min\{\mu([w]_0^+)\colon w\in\mathcal{L}_{|\beta_0|\cdot(m+1)}(X)\}}{\max\{\mu([w]_0^+)\colon w\in\mathcal{L}_{|\beta_0|\cdot(m+1)}(X)\}}<4C_2 
$$ 
\end{lemma} 

We note that while this seems like a long list of assumptions on the shift $Z$, these hypotheses are satisfied by the shifts to which we apply our inductive construction.  Starting with the system defined in Section~\ref{previous}.  In our application, 
the shift $Z$ is the shift $Z$ from the preceding section and the word $v$ is taken to be $xxxy$.

\begin{proof} 
Let $\mu$ be an ergodic measure supported on $X$ and let $x\in X$ 
be a generic point for the measure $\mu$.  
Choose $u_1,u_2\in\mathcal{L}_{(m+1)|\beta_0|}(X)$ such that 
$$ 
\mu([u_1]_0^+)=\max\{\mu([w]_0^+)\colon w\in\mathcal{L}_{m|\beta_0|}(X)\} 
$$ 
and 
$$ 
\mu([u_2]_0^+)=\min\{\mu([w]_0^+)\colon w\in\mathcal{L}_{m|\beta_0|}(X)\} 
$$ 
Since $\mu$ is ergodic, 
$$ 
\mu([u_1]_0^+)=\lim_{n\to\infty}\frac{1}{n}\sum_{k=0}^{n-1}1_{[u_1]_0^+}(\sigma^kx) 
$$ 
and 
$$ 
\mu([u_2]_0^+)=\lim_{n\to\infty}\frac{1}{n}\sum_{k=0}^{n-1}1_{[u_2]_0^+}(\sigma^kx). 
$$ 
We analyze occurrences of $u_1$ in $x$, noting that the same analysis applies for occurrences of $u_2$.
Any occurrence of $u_1$  in $x$ occurs in a concatenation of the words $\beta_0$ and $\beta_1$, and since $m\geq2$, the value of $d\in\{0,1,\dots,|\beta_0|-1\}$ where the beginnings of these concatenated words occur (starting from the left edge of $u_1$) are determined entirely by the word $u_1$ itself.  (In other words, if $u_1$ occurs in two different places within $x$, the value of $d$ can not change between occurrences of $u_1$, 
 as by considering where the word $v$ is located within $u_1$, we can locate the beginnings.) 
 Thus each occurrence of $u_1$ occurs as a subword of a concatenation of at most 
 $m+1$ of the words $\beta_0$ and $\beta_1$ and that all but (perhaps) the first and last of the words $\beta_0$ and $\beta_1$ can be determined from the word $u_1$.   
 Therefore there are at most four ways to concatenate the words $\beta_0$ and $\beta_1$ such that $u_1$ 
 occurs as a subword, corresponding to the ambiguity of the edge (first and last) concatenated words and that there are at 
 most two choices for each of these edge words. 
 This means that the asymptotic frequency with which $u_1$ occurs as a subword of $x$ is at least the frequency with which the sequence of $0$'s and $1$'s giving the indices of the $m-1$ 
 non-edge $\beta_i$ words occur in the element of $Z$ corresponding to $x$, it is at most four times as frequent.  
 Since the ratio of the least to the most frequently occurring words of length $m-1$ in $\mathcal{L}(X)$ is 
linear and bounded 
$C_1$ and $C_2$ for all ergodic measures supported on $X$, this also holds for all elements of $X$.  Namely, if 
there were a point in $X$ not satisfying these bounds, then via a standard argument of passing to a limit of the empirical  measures and taking an ergodic component, we would contradict the bounds imposed by $C_1$ and $C_2$.   Thus, we conclude that 
\begin{equation*} 
\frac{C_1}{4}<\frac{\mu([u_1]_0^+)}{\mu([u_2]_0^+)}<4C_2. \qedhere
\end{equation*}
\end{proof} 

\subsection{Induction}
Let $\{a_m\}_{m=1}^{\infty}$ be a subexponentially growing sequence of positive integers, meaning 
$$ 
\limsup_{m\to\infty}\frac{1}{m}\log a_m=0. 
$$ 
We inductively construct a sequence of shifts $X_1,X_2,X_3,\dots$ and ultimately define our subshift $Y$ from Theorem~\ref{th:main}.  Let $n_1\in\N$ be the smallest integer that satisfies~\eqref{eq:exist-n} and let, by taking $n=n_1$ in~\eqref{eq:fixed}, let $X_1\subseteq\{0,1\}^{\Z}$ be the subshift constructed at the end of Section~\ref{previous} (where it was called $Z$).  Let $N$ be the parameter arising in Section~\ref{previous} and let $0_1$ and $1_1$ be the words defined in Equations~\eqref{eq:0*} and ~\eqref{eq:1*} constructed from the sequence $\{m\cdot a_m\}_{m=1}^{\infty}$ and $B\geq N$ in~\eqref{eq:bound} and let $N_1\geq B$ be the parameter $m$ in the sentence following~\eqref{eq:bound}.  Then we have $P_{X_1}(N_1)>N_1\cdot a_{N_1}$. 

Now suppose we have constructed a nested sequence of subshifts 
$$ 
X_1\supseteq X_2\supseteq X_3\supseteq\cdots\supseteq X_k, 
$$ 
a sequence of positive integers $n_1<N_1<n_2<N_2<\cdots<n_k<N_k$, and a sequence of words $0_1,1_1,0_2,1_2,\dots,0_k,1_k$ where $0_i,1_i\in\mathcal{L}(X_i)$ are two words of equal length (and this common length is at least $N_k$).  We suppose that for each $i\leq k$, $X_i$ is the subshift obtained by taking all possible bi-infinite concatenations of the words $0_i$ and $1_i$.  Suppose further that for any $i\leq k$ and any $j\leq i$ we have $P_{X_i}(N_j)>N_j\cdot a_{N_j}$ and 
$$ 
\frac{1}{16}<\frac{\min\{\mu([w]_0^+)\colon w\in\mathcal{L}_{n_j}(X_i)\}}{\max\{\mu([w]_0^+)\colon w\in\mathcal{L}_{n_j}(X_i)\}}<1. 
$$ 
Take $n_{k+1}>|0_k|$ to be an integer satisfying~\eqref{eq:exist-n} and let $N_{k+1} =  |0_*|$ when $n$ 
is chosen to be $n_{k+1}$ in~\eqref{eq:fixed}.  
We apply Lemma~\ref{lem:induction} with $Z=X_{\alpha}$, $m=n_{k+1}$, $C_1=1/4$, $C_2=3/4$, $\beta_0=0_k$, and $\beta_1=1_k$ to produce a new subshift $X_{k+1}$.  
Note that every element of $X_{k+1}$ can be written as a bi-infinite concatenation of $0_k$ and $1_k$, and so $X_{k+1}\subseteq X_k$.  Let $0_{k+1}$ and $1_{k+1}$ be the words in $\mathcal{L}(X_{k+1})$ by ``inflating'' $0_*$ and $1_*$ from Section~\ref{previous}, with parameter $m=n_{k+1}$, using the words $\beta_0=0_k$ and $\beta_1=1_k$. 

Note that since $N_{k+1} = |0_*|$, by
Equation~\eqref{eq:blank} and taking $b_n =na_n$, we have guaranteed that  $P_{X_{k+1}}(N_j)\geq N_ja_{N_j}$ for all $j\leq k+1$.  Finally, by Lemmas~\ref{lem:meas-concat} and~\ref{lem:induction}, we have 
$$ 
\frac{1}{16}<\frac{\min\{\mu([w]_0^+)\colon w\in\mathcal{L}_{n_j}(X_{k+1})\}}{\max\{\mu([w]_0^+)\colon w\in\mathcal{L}_{n_j}(X_{k+1})\}}<1 
$$ 
for all $j\leq k+1$.  

Finally, define 
$$ 
X:=\bigcap_{k=1}^{\infty} X_k.  
$$ 
Then  since the subshifts are nested, $X$ is nonempty.  Since each pattern occurs syndetically, the system $X$ 
is minimal.  Finally, by construction, we obtain a 
subshift satisfying $P_X(N_j)>N_ja_{N_j}$ for all $j\in\N$, in particular 
$$
\limsup_{n\to\infty}\frac{P_X(n)}{a_n}=\infty, 
$$ 
and such that
$$ 
\frac{1}{16}<\frac{\min\{\mu([w]_0^+)\colon w\in\mathcal{L}_{n_j}(X)\}}{\max\{\mu([w]_0^+)\colon w\in\mathcal{L}_{n_j}(X)\}}<1 
$$ 
for all $j\in\N$.  This concludes the proof of Theorem~\ref{th:main}.

\end{document}